\documentclass[reqno]{amsart}
\usepackage{amsfonts}
\usepackage{graphicx}
\usepackage{amscd}
\usepackage{amsmath}
\usepackage{amssymb}
\usepackage{latexsym}
\usepackage{hyperref}
\usepackage[all]{xy}
\usepackage{color}
\usepackage{mathrsfs}
\usepackage[latin1]{inputenc}
\newtheorem{exem}{\bf Example}
\newtheorem{lemma}{\bf Lemma}

\newtheorem{proposition}{\bf Proposition}
\newtheorem{remark}{Remark}

\newtheorem{theorem}{\bf Theorem}
\newtheorem*{theorem*}{\bf Theorem}
\newtheorem*{thma}{\bf Theorem\,A}
\newtheorem*{thmb}{\bf Theorem\,B}
\newtheorem*{thmc}{\bf Theorem\,C}
\newtheorem*{cor}{\bf Corollary}
\newtheorem*{Codazzi}{\bf The Codazzi Equation}
\newtheorem*{Gauss}{\bf The Gauss Equation}
\newtheorem*{GF}{\bf The Gauss Formula}
\newtheorem*{WF}{\bf The Weingarten Formula}
\newtheorem*{WE}{\bf The Weingarten Equation}
\numberwithin{equation}{section}

\usepackage{xcolor}
\begin{document}

\title[On Ricci almost solitons arising from conformal vector fields]{On Ricci almost solitons arising from conformal vector fields}

\author[J.N.V. Gomes, J.F.B. Pereira]{J.N.V. Gomes$^1$, J.F.B. Pereira$^2$}
\author[D.M. Tsonev]{D.M. Tsonev$^3$}

\address{$^{1}$Departamento de Matem\'atica, Universidade Federal de S\~ao Carlos, Rod. Washington Lu\'is, Km 235, 13565-905 São Carlos, S\~ao Paulo, Brazil.}
\address{$^{2,3}$Departamento de Matemática, Universidade Federal do Amazonas, Av. General Rodrigo Octávio, 6200, 69080-900 Manaus, Amazonas, Brazil.}

\email{$^1$jnvgomes@ufscar.br }
\email{$^2$ joao\_jou@hotmail.com }
\email{$^3$ tsonevdm@ufam.edu.br}

\urladdr{$^1$https://www.ufscar.br}
\urladdr{$^{2,3}$https://www.ufam.edu.br}
\keywords{Ricci almost solitons, Hypersufaces, Conformal vector fields}
\subjclass[2010]{Primary 53C15; Secondary 53C17, 53C25.}

\begin{abstract}
Let $\overline{M}^{n+1}$ be a semi-Riemannian manifold of constant sectional curvature, and endowed with a conformal vector field .  Consider a Riemannian manifold  $M^n$, isometrically immersed into  $\overline{M}^{n+1}$.  With these hypotheses in mind, the ultimate goal of this paper is to investigate the intimate relationship between conformal vector fields on  $\overline{M}^{n+1}$ and the Ricci almost soliton structure on $M^n$. Assuming that the latter manifold is connected and totally umbilic, we prove that  it is naturally given a Ricci almost soliton structure by means of the tangential part of a conformal vector field on the ambient manifold. This result is rather general in nature, and few concrete examples are worked out to illustrate its true power. Furthermore, with Tashiro's theorem in mind, we reach the climax of this paper with a classification of a class of Riemannian hypersurfaces that naturally inherit their Ricci almost soliton structure by the existence of a conformal vector field on the ambient manifold.
 \end{abstract}

	\maketitle

%=============================================================================================
%                                                                    INTRODUCTION
%=============================================================================================

\section{Introduction}

The ultimate goal of this article is twofold. Firstly, we show how  conformal vector fields on a semi-Riemannian manifold naturally give rise to a 
Ricci almost soliton structure onto isometrically immersed Riemannian hypersurfaces. Secondly, we classify the latter. Our motivation stems from a recent work by Aquino, de Lima, and the first author of the present article \cite{Pacific}. The latter three authors characterised gradient Ricci almost solitons immersed in either the hyperbolic space, the de Sitter space or the anti-de Sitter space. Albeit our approach is different, a deeper understanding of the present work can be incontrovertibly  achieved by comparing it with  \cite{Pacific}. For this reason, we offer a brief discussion on this matter in the concluding section.

We reach our final goal by blending together ideas from two articles. These articles were written in different epochs and contexts.
Remarkably, results from both articles beautifully manifest themselves in our investigation. Wherefore, they ought to be thought of as part of both our motivation and inspiration. In 1965, while studying conformal transformations on Riemannian manifolds, Tashiro ~\cite{Tashiro} proved his celebrated theorem which in essence revealed the structure of the manifolds admitting scalar concircular fields. More precisely, Tashiro classified all complete Riemannian manifolds of dimension $n\geq 2$ that admit a special  scalar concircular field. In this article we show that all but one of the manifolds in Tashiro's classification, provided that they are isometrically immersed into a semi-Riemannian manifold,  can be endowed with a Ricci almost soliton structure arising from a conformal vector field on the ambient space. In 2002,  Kim et al. \cite{Kim} characterised  conformal vector fields on totally umbilic hypersurfaces of semi-Riemannian space forms. It was shown in their paper, among other things, that any conformal vector field on a connected totally umbilic  hypersurface of a semi-Riemannian space form can be obtained as the tangent part of a conformal vector field on the ambient space. In this paper, we translate this latter result into the language of Ricci almost solitons. Strictly speaking, we prove that the Ricci almost soliton structure on a connected totally umbilic hypersurface immersed into a semi-Riemannian manifold arises naturally  from the tangential part of a conformal vector field on the ambient space.

It must be evident by now that this article achieves two things. First, it studies an interesting class of Ricci almost solitons. 
Second, an intimate relationship between the conceptions of umbilic hypersurface, conformal vector fields and Ricci almost soliton structure is perceived. Thus, this article will be erected upon three pillar hypotheses, namely, the suppositions  of {\it isometric immersion},  {\it completeness},  and {\it umbilicity}. While the latter hypothesis will sometimes not be required, the former two will be absolutely indispensable.

The upshot of our investigation is the following classification theorem.

\begin{theorem*}
Let $(M^n,g)$ be a complete Riemannian manifold isometrically immersed into a  semi-Riemannian manifold $(\overline{M}^{n+1}, \overline{g})$ of constant sectional curvature endowed with a conformal vector field $\overline{V}$.  Let  $\sigma$ be the conformal factor of $\overline{V}$ and $\lambda$ be the soliton function. Suppose that $\psi=\sigma|_{M^n}-\lambda$ is non-zero in some dense subset of $M^n$ and $Ric = \mu \mathcal{A}$, for some smooth function $\mu$ on $M^n$. Then, the Ricci almost soliton arising from the conformal vector field $\overline{V}$ is one of the following manifolds: 
\newline\\
A)\,\, an Euclidean space $\mathbb{R}^n$;
\newline\\
B)\,\, a sphere $\mathbb{S}^n$;
\newline\\
C)\,\, a hyperbolic space $\mathbb{H}^n$;
\newline\\
D)\,\, a pseudo-hyperbolic space of zero or negative type.
\end{theorem*}

The paper is structured as follows. We begin in Section~\ref{sec2} by defining and/or commenting upon all the concepts required to prepare the ground for our investigation. Thus, we give a swift panoramic view of basic {\it semi-Riemmanian geometry}, {\it submanifold theory}, the theory of {\it Ricci almost solitons} and {\it concircular fields}.  In Section~\ref{sec3}, we proceed with the proofs of four lemmas which not only are  employed in the proofs of the main  results of this paper but also are interesting in their own right.  In Section~\ref{sec4}, we venture into the detail of our enquiry which results in Theorems \ref{thm1} and \ref{thm2}. Theorem~\ref{thm1} yields that the Ricci almost soliton structure on a totally umbilic hypersurface is naturally given by the tangent component of some conformal vector field on the ambient space. Theorem~\ref{thm2} goes a little beyond in the following sense.  Assuming that $M^n$ is already endowed with Ricci almost soliton structure by virtue of Theorem~\ref{thm1} as well as imposing two further hypotheses, Theorem~\ref{thm2} affirms that the angle function is a special scalar concircular field on $M^n$. In Section~\ref{sec5}, we work out four examples that shed some more light onto the geometric structure of the class of the Ricci almost solitons under consideration herein. We conclude our quest, in Section~\ref{conclusions}, by perceiving the verity of the classification theorem stated above.

%===============================================================================================
%                                                                                   SETTING THE CONTEXT
%===============================================================================================

\section{Setting the context}\label{sec2}
The context of this article is a blend of few subareas of differential geometry. For this reason, and for the sake of clarity, we swiftly sketch in this preliminary section some basic conceptions from the {\it semi-Riemmanian geometry}, {\it submanifold theory}, and the theory of {\it Ricci almost solitons}. We also discuss the notions of  {\it concircular field}  and  {\it conformal field}, and recall a celebrated theorem by Yoshihiro Tashiro. 

%===============================================================================================
%                                      SEMI-RIEMANNIAN GEOMETRY AND SUBMANIFOLD THEORY
%===============================================================================================

\subsection{A hint of semi-Riemmanian geometry and Submanifold theory}In this section, we fix notation as well as recollect some important definitions which will be used throughout. We also recall the Codazzi and Gauss equations, which will play a fundamental role henceforth.  For greater detail, the reader may wish to consult the book ``{\it Semi-Riemannian geometry with applications to relativity}" by Barrett O'Neill \cite{O'Neill}.   

 We shall write, with an exception to tangent vector fields and inner product, an upper  bar to distinguish the semi-Riemannian from the Riemannian entities. This slight abuse of notation is meant to achieve a better visual aesthetics of the formulae. Thus, the reader should  be able to discern from the context whether $X,Y,Z$ and $\langle \,\cdot,\,\cdot \rangle$ are to be considered in the Riemannian or semi-Riemannian case. Otherwise, we write $(M^n,g)$ for an $n$-dimensional Riemannian manifold with metric $g$ and  $(\overline{M}^{n+1}, \overline{g})$ for an $(n+1)$ - dimensional semi-Riemannian manifold with metric $\overline{g}$. The corresponding Levi-Civita connections will be denoted by $\nabla$ and $\overline{\nabla}$, respectively, and an isometric immersion of  the former manifold into the latter by $f: (M^n,g) \hookrightarrow (\overline{M}^{n+1}, \overline{g})$.  Let  $\mathfrak{X}(M^n)$ be the ring of smooth tangent vector fields on $M^n$, $N$ be the field of unit normal vectors  on $M^n$,  $\mathcal{A}: \mathfrak{X}(M^n) \rightarrow \mathfrak{X}(M^n)$ be the Weingarten operator and  $\alpha : \mathfrak{X}(M^n) \times \mathfrak{X}(M^n) \rightarrow \mathfrak{X}(M^n)^{\perp}$ the second fundamental form of the immersion $f$. The following relations are well-known in Submanifold theory and will be frequently used, without a proper reference, in various computations in this paper.
\begin{GF}
 $$\overline{\nabla}_XY = \nabla_XY + \alpha (X,Y).$$ 
\end{GF}   
\begin{WF}
   $$\overline{\nabla}_XN = - \mathcal{A} X.$$
 \end{WF}
  \begin{WE} 
  $$\overline{g}(\alpha (X,Y),N) = g(\mathcal{A} X, Y).$$
\end{WE}
The following standard notation is also to be adopted. We write $\lbrace e_1,\ldots,e_n,N \rbrace$ for the frame adapted to the immersion $f$, where $N$ is the normal vector field to $M^n$. The mean curvature vector and the mean curvature function of the immersion $f$ are defined by $\vec{H}=\frac{1}{n}tr(\alpha)N$ and $H = \overline{g}(\vec{H},N)$, respectively. By writing $\epsilon_N = \overline{g}(N,N) = \pm 1$, we have a more practical definition of the mean curvature function, namely, $H=\epsilon_N\frac{1}{n}tr(\alpha)$. Both the curvature tensor $\overline{R}$ and the sectional curvature $\overline{K}$ of a  semi-Riemannian manifold $\overline{M}^{n+1}$  are defined exactly  in the same way as in the Riemannian case (clearly with respect to $\overline{g}$ and $\overline{\nabla}$), and therefore we shall not explicitly define them herein. We write $\overline{M}^{n+1}_c$ to indicate that the sectional curvature of $\overline{M}^{n+1}$ is constant\footnote{This is the last time we abuse notation and do not use the upper bar to indicate a geometric entity on the ambient manifold.}. Considering an orthonormal basis $\{e_1,\ldots,e_{n+1}\}\subset T_p\overline{M}^{n+1}$ we write $\epsilon_i=\langle e_i,e_i \rangle=\pm 1$. Then, for any $X,Y\in T_p\overline{M}^{n+1}$, the tensor of the Ricci curvature of $\overline{M}^{n+1}$ is defined by
$$\overline{Ric}(X,Y)=\displaystyle\sum_{i=1}^{n+1}\epsilon_i\langle \overline{R}(e_i,X)e_i,Y\rangle=\displaystyle\sum_{i=1}^{n+1}\epsilon_i \langle \overline{R}(e_i,Y)e_i,X\rangle.$$

Both the Codazzi and the Gauss equations for  the immersion $f$ will be of great importance for our considerations. In the 
semi-Riemannian context, they read as follows. 

\begin{Codazzi}
\begin{eqnarray}\nonumber\label{eqCodazzi}
\nonumber(\overline{R}(X,Y)Z)^{\perp} = (\nabla^{\perp}_Y \alpha)(X,Z) - (\nabla^{\perp}_X \alpha)(Y,Z).
\end{eqnarray}
\end{Codazzi}
%\vspace{0,5cm}
\begin{Gauss}
$$\overline{g}(\overline{R}(X,Y)Z,W) = \overline{g}(R(X,Y)Z,W) - \overline{g}(\alpha(X,Z),\alpha(Y,W))
+ \overline{g}(\alpha(Y,Z),\alpha(X,W)).$$
\end{Gauss}

Contracting the Gauss equation  one obtains
\begin{equation}\label{eqGausscontraida}
\overline{Ric} = Ric +\epsilon_N\Omega_N-nH \mathcal{A} +\epsilon_N\mathcal{A}^2,
\end{equation}
where $\Omega_N (\cdot , \cdot) = \overline{g}(\overline{R}(\cdot,N)\cdot,N).$ Now, if the ambient space $(\overline{M}_{c}^{n+1},\overline{g})$ is a semi-Riemannian manifold with constant sectional curvature $c$, then the well-known relation
\begin{equation*}
\langle \overline{R}(X,Y)Z,W \rangle = c \left\{\langle X,Z \rangle \langle Y,W \rangle - \langle Y,Z \rangle \langle X,W \rangle\right\},\end{equation*}
for all $X,Y,Z,W \in T_p\overline{M}^{n+1}$ implies that $\Omega_N=\epsilon_Ncg$ and $\overline{Ric}=nc\overline{g}$. Thus, equation  \eqref{eqGausscontraida} reduces to
\begin{equation}\label{EqRicM}
Ric=c(n-1)g +nH\mathcal{A}-\epsilon_N\mathcal{A}^2.
\end{equation}
Notice that it is this latter equation that will be of paramount importance henceforth.

%===============================================================================================
%                                                                      RICCI ALMOST SOLITONS 
%===============================================================================================

\subsection{Ricci Almost Solitons} As of the time of writing, Ricci almost solitons are on the one hand still being only recently discovered structures, and on the other, are already drawing the attention of more researchers. Ricci almost solitons were first defined in 2011 by Pigola, Rigoli, Rimoldi and Setti  in an article bearing the same name \cite{prrs}.  At first glance the definition given in \cite{prrs} was very natural to think of and deviated only very little from the definition of a Ricci soliton. And while one could rightly say, in terms of definition, that Ricci almost solitons naturally included Ricci solitons, it was quickly perceived that the former were rather different in nature than the latter. For this reason, we should like to swiftly paint with a very broad brush the landscape which is constituted by  Ricci and Ricci almost solitons.

This is rather a long story which goes back to 1982 and the seminal paper by Richard Hamilton \cite{Hamilton} in which he defined the famous Ricci flow as a parabolic type of PDE with respect to the metric, namely,
$$\dfrac{\partial\,g}{\partial t}=-2\, Ric.$$
 The short time existence and the uniqueness of the solutions was proven by Hamilton, and the self-similar solutions to the flow  were called  Ricci solitons. It was quickly understood that the Ricci flow contained a lot of geometric and topological information and many people thoroughly studied both the flow and the Ricci solitons. Many would argue that the apex of the studies of the Ricci flow  is undeniably Perelman's proof of the Poincare conjecture~\cite{p1,p2,p3}. It is also worth noting that in 2010, a year before the definition of Ricci almost solitons appeared in published form, Huai-Dong Cao had already published the now celebrated  survey  ``Recent Progress on Ricci solitons" \cite{cao}.  The assiduous studies on the matter led to an alternative definition of a Ricci soliton via a tensorial equation in terms  of the Ricci curvature, the metric and the Lie derivative of the metric. It was this latter definition that was the stepping stone for Pigola, Rigoli, Rimoldi and Setti.  
 
A Riemannian metric $g$ on a smooth manifold $M^n$ is called a {\it Ricci almost soliton} if the Ricci tensor of $g$ satisfies the equation   
\begin{equation}\label{ars}
Ric + \frac{1}{2}\mathcal{L}_Xg = \lambda g,
\end{equation}
for some vector field $X$ and a {\it soliton function}  $\lambda : M^n \longrightarrow \Bbb{R}$. By writing the triple $(M^n,g,X)$ we shall understand that the manifold $M^n$ is given a Ricci almost soliton structure. When $\lambda$ is constant, then the metric $g$ in equation~\eqref{ars} is referred to as a Ricci soliton. Since Ricci almost solitons contain Ricci solitons as a special case, we say that the Ricci almost soliton is {\it proper} if the soliton function $\lambda$ is non-constant. In this paper we shall only consider proper Ricci almost solitons, and for the sake of brevity we shall in fact omit the adjective ``proper".

The following very important remark is due. The parameters in \eqref{ars} are $g$ and $X$, while the function $\lambda$ is readily obtained by taking trace of this latter equation. Notice that, as we work exclusively with isometric immersions, we have that the metric $g$ is fixed and therefore it is the field $X$ which virtually determines the Ricci almost soliton structure $(M^n,g,X)$. This fact permeates in all our considerations henceforth and  motivates the inquiry for vector fields which will result in the Ricci almost soliton structure on $M^n$. In addition, when $X$ is a gradient vector field of some differentiable function $f: M^n \longrightarrow \Bbb{R}$, the metric $g$ is known as a {\it gradient Ricci almost soliton}. In this case equation~\eqref{ars}  reads
\begin{equation*}
Ric + \mathrm{Hess}\, f = \lambda g,
\end{equation*}
where $\mathrm{Hess}\, f$ denotes the Hessian of the potential function $f$. It must be noticed at this juncture that henceforth we shall use the term {\it potential field} to exclusively refer to a vector field which defines the gradient Ricci almost soliton structure on $M^n$. 

Let us finish our brief discussion on the Ricci almost solitons by highlighting, alas without exhausting all,  some papers in the field which we consider important. In 2012, Barros and Ribeiro  \cite{br2} gave some characterisations for compact Ricci almost solitons by proving some rigidity results as well as deriving some structural equations. In 2014, Barros, Batista and Ribeiro \cite{brb}  proved that any compact nontrivial Ricci almost soliton with constant scalar curvature is isometric to the Euclidean sphere and it is necessarily gradient. During the same time Catino and Mazzieri uploaded a preprint on the arXiv studying Gradient Einstein solitons which was later published in 2016 \cite{Catino-Mazzieri}. They considered the so called Ricci - Bourguignon flow 
$$\dfrac{\partial\,g}{\partial t}=-2 (Ric-\rho Rg),$$
where $R$ is the scalar curvature and $\rho$ some real constant. They defined {\it gradient $\rho$- Einstein solitons} by the 
equation
$$Ric+\mathrm{Hess}\, f=\rho Rg+\lambda g$$
and studied some of their properties. Notice that whereas $\rho, \lambda$ are constants, the scalar curvature $R$ is generally not. One already sees in \cite{Catino-Mazzieri} that some of the properties of these gradient $\rho$- Einstein solitons cannot be valid for gradient Ricci solitons. At this point, the reader must have already observed that the definition of a gradient $\rho$- Einstein soliton is a very special case of the general definition of Ricci almost soliton~\eqref{ars}. Now, with all that is known about Ricci solitons in mind, a natural question is whether or not Ricci almost solitons are also solutions to some geometric flow. The answer to this question is in the negative as one readily perceives in Catino et al.~\cite{Catino2}. In this latter paper it is proven the uniqueness and the short time existence of the solution of the Ricci - Bourguignon flow provided that $\rho < \dfrac{1}{2} (n-1) $ and $M$ be $n$ - dimensional compact manifold. This restriction for the parameter $\rho$ clearly means that  Ricci almost solitons do not arise in general as solutions to some appropriate geometric flow and therefore the best approach known by far to study these structures is indeed stepping upon the tensorial equation definition~\eqref{ars}. 
 
 In 2017 were published another two important papers studying Ricci almost solitons. One, as already mentioned in the introduction, was on the characterisation of immersed gradient Ricci almost solitons \cite{Pacific} and the other was a work by   
Calvi\~{n}o - Louzao et al.  \cite{EMER}.  In the latter paper, homogeneous Ricci almost solitons  were studied and it was proved that  
 a locally homogeneous proper Ricci almost soliton is either of constant sectional curvature or locally isometric to a product $R\times N(c)$, where $N(c)$ is a space of constant curvature. With this, we should like to  finish painting the panoramic picture of  the realm of Ricci almost solitons. The interested reader may take care to consult the literature for further details.

%===============================================================================================
%                    CONFORMAL VECTOR FIELDS AND TOTALLY UMBILIC HYPERSURFACES
%===============================================================================================

\subsection{Conformal vector fields and totally umbilic hypersurfaces } Given a semi-Riemannian manifold $(\overline{M}^{n+1},\overline{g})$, a vector field $X\in\mathfrak{X}(\overline{M}^{n+1})$ is called {\it conformal} with respect to the metric $\overline{g}$ with conformal factor $\sigma \in C^{\infty}(\overline{M}^{n+1})$, when the following equation holds
	\begin{equation*}\mathcal{L}_X\overline{g}=2 \sigma \overline{g}.
	\end{equation*} 
Notice that from now on we shall reserve the symbol $\overline{V}$ to exclusively denote throughout conformal vector fields on the ambient manifold $\overline{M}^{n+1}$. If the first and second fundamental forms of a submanifold of a semi-Riemannian manifolds are proportional, then the submanifold is called {\it totally umbilic}. Notice that, in practice, totally umbilical  tantamounts to the identity $\mathcal{A}=\epsilon_NH I$, where $I$ is the identity operator. 	 

Of fundamental importance for our method is a paper by  Kim et al. \cite{Kim} in which they proved the following theorem.
\begin{thma}\label{Kimthm}[Kim, Kim, Kim, Park~\cite{Kim}]
Let $M^n$ be a connected totally umbilic hypersurface
of a semi-Riemannian space form $\overline{M}^{n+1}_c$. Then any conformal vector 
field on $M^n$ can be obtained as the tangential part $V^{\top}$ of a conformal vector field $V$ on  $\overline{M}^{n+1}_c$. Furthermore, for any conformal vector field $W$ on $M^n$ there exists a unique conformal vector field $V$ on $\overline{M}^{n+1}_c$ which satisfies $V|_{M^n} = W$.
\end{thma}
This latter theorem is the prototype of our Theorem~\ref{thm1} and will essentially guarantee that the Ricci almost soliton structure in our general context is determined by the tangent part of a given conformal field on the ambient manifold. Furthermore, the complete description of conformal vector fields on semi-Riemannian space forms, also  given in \cite{Kim}, will enable us to elaborate examples of gradient Ricci almost solitons isometrically immersed into a semi-Riemannian space form.

%===============================================================================================
%                                                                     TASHIRO'S THEOREM
%===============================================================================================

\subsection{Concircular fields and Tashiro's theorem} The conception of a concircular field is classical in differential geometry.  Given a Riemannian manifold $(M^n,g)$ we say that a function $\rho \in C^{\infty}(M^n)$ is
a {\it scalar concircular field} if it satisfies the equation 
\begin{equation}\label{defconc}
\mathrm{Hess}\,\rho = \phi g,
\end{equation}
for some $\phi \in C^{\infty}(M^n)$, called the {\it characteristic function}. It is worth recalling that the term {\it concircular} originated from conformal transformations that preserve the geodesic circles, called concircular transformations themselves, see Fialkow~\cite{Fialkow} and Yano~\cite{yano}. If the characteristic function $\phi$ is of the form $\phi = - k\rho + b$, where $b$ and $k$ are real constants, we say that $\rho$ is a {\it special scalar concircular field}. Tashiro thoroughly studied these latter and proved his celebrated theorem.

\begin{thmb}\label{Tashiro}[Tashiro~\cite{Tashiro}]
Let  $M^n$ be a complete Riemannian manifold of dimension $n \geq 2$ and suppose it admits a special scalar concircular  field $\rho$ satisfying the 
equation
\begin{equation*}
\mathrm{Hess}\,\rho = (-k\rho+b)g
\end{equation*}
where $b, k \in \Bbb{R}$ and $g$ is the Riemannian metric on $M^n$. Then, $M^n$ is one of the following manifolds:
\begin{enumerate}
	\item If $b=k=0$, the direct product $\widetilde{M} \times I$ of $(n-1)$ - dimensional complete Riemannian manifold $\widetilde{M}$ with a straight line $I$,
	\item If $k=0$ but $b \neq 0$, a Euclidean space,
	\item If $k=-c^2<0$ and $\rho$ does not have any isolated stationary points, a pseudo-hyperbolic space of zero or negative type;
	\item If $k=-c^2<0$ and $\rho$ has one isolated stationary point, a  hyperbolic space of curvature $-c^2$, and
	\item If $k=c^2>0$, a spherical space of curvature $c^2$, where $c$ is a positive constant.
\end{enumerate}
\end{thmb}

At this juncture, we venture for a brief discussion on the definitions of  {\it pseudo-hyperbolic spaces of zero and negative type}. Following Tashiro's paper \cite{Tashiro}, we have the following general picture. Given a scalar concircular field $\rho$, it is known that the trajectories of the vector field $\rho^k=\rho_{\lambda}g^{\lambda k}$ are geodesic arcs except at stationary points of $\rho$, and a geodesic curve in $M^n$ containing such an arc is called a $\rho$-curve. Consider now a special scalar concircular field $\rho$ satisfying the equation
 $$\mathrm{Hess}\,\rho=(-k\rho+b)g.$$
 Given a geodesic curve $l$ with arc length $s$ one can reduce the latter equation to the following ordinary differential equation
 \begin{equation}\label{ode1}
 \dfrac{\mathrm{d}^2\,\rho}{\mathrm{d}\,s^2}+k\rho=b.
 \end{equation}
 Now, taking into account the signature of the characteristic constant $k$ one needs to consider the following three cases 
 \begin{equation*}
 k=0\,\,\, (I), \,\,\,\,\,\, k=-c^2\,\,\, (II), \,\,\,\,\,\, k=c^2\,\,\, (III), 
 \end{equation*}
 where $c$ is a positive constant. It is known that, choosing suitably the arc length $s$, the solution of \eqref{ode1} is given by
 
 \begin{equation}
 \rho(s) =
    \begin{cases}
      (I,A) & as\,\, (b=0), \\
      (I,B) & \dfrac{1}{2}bs^2+a\,\, (b\neq 0),\\
       (II, A_0) & a\,\mathrm{exp}\,cs-\dfrac{b}{c^2}\, ,\\
        (II, A_{\_}) & a\, \mathrm{sinh}\, cs-\dfrac{b}{c^2}\, ,\\
         (II, B) & a\, \mathrm{cosh}\, cs - \dfrac{b}{c^2}\, , \\
      (III) & a\, \mathrm{cos}\, cs + \dfrac{b}{c^2}\, ,
    \end{cases}       
\end{equation}
 $a$ being an arbitrary constant. It is known that if $M^n$ is complete, then in the cases $(I,A)$, $(II,A_0)$ and $(II,A_{\_})$ there are no stationary points, in the cases $(I,B)$ and $(II,B)$ there is one stationary point corresponding to $s=0$, and in case $(III)$ there are two stationary points corresponding to $s=0$ and $s=\dfrac{\pi}{c}$. Now, in the cases $(I,A)$, $(II,A_0)$ and $(II,A_{\_})$ one chooses the arc length $s$ of $\rho$-curves so that the points corresponding to $s=0$ lie on the same $\rho$-hypersurface. As a consequence the coefficient $a$ is the same for all $\rho$-curves. Taking $s$ as the $n^{th}$ coordinate $u^n$ it follows from Lemma 1.2 in \cite{Tashiro}  that the metric form of $M^n$ is given by    
 \begin{equation}\label{phm}
 \mathrm{d}\,s^2 =
    \begin{cases}
      (I,A) & a^2\,\widetilde{\mathrm{d}\,s^2}+(\mathrm{d}\,u^n)^2, \\
      (II,A_0) & (ac\,\mathrm{exp}\,cu^n)^2\widetilde{\mathrm{d}\,s^2}+(\mathrm{d}\,u^n)^2,\\
      (II, A_{\_}) & (ac\,\mathrm{cosh}\,cu^n)^2 \widetilde{\mathrm{d}\,s^2}+(\mathrm{d}\,u^n)^2 .
    \end{cases}       
\end{equation}
Notice that $\widetilde{\mathrm{d}\,s^2}$ is the metric form on $\widetilde{M}$. A complete Riemannian manifold, which is topologically a product $\widetilde{M}\times I$ and has the metric of the form $(II,A_0)$ or $(II,A_{\_})$ of \eqref{phm}, is called {\it pseudo-hyperbolic space of zero} or {\it negative type}, respectively. 
 
%===============================================================================================
%                                                             FOUR LEMMAS
%===============================================================================================

\section{Four  lemmas}\label{sec3}
In this section we prove four very important lemmas. The nature of the  first two is rather conceptual, as opposed to the technical character of the latter two. Before we begin with their proofs, however, the following important comments are due. Firstly, these lemmas will remain valid for semi-Riemannian hypersurfaces $M^n$ isometrically immersed into a semi-Riemannian manifold $\overline{M}^{n+1}$ up to a signature sign. Secondly, all these lemmas do not require the hypothesis of umbilicity, so this pillar hypothesis of ours will be dropped out in this section.  Finally, with both Theorem A and what we are yet to prove in Section~\ref{sec4} in mind, we should like to set, once and for all, the following notation. We shall always write henceforth $V$ for the tangential part of a conformal field $\overline{V}$ on $\overline{M}^{n+1}$ and $C:=\overline{g}(\overline{V},N)$ for the angle function.

%================================================================================================
%                                                                               LEMMA 1
%================================================================================================

\begin{lemma}\label{lemmaDerivadaLie}
Let $(M^n,g)$ be a Riemannian manifold isometrically immersed into a semi-Riemannian manifold $(\overline{M}^{n+1}, \overline{g})$, and $\overline{V}$ be a vector field  on $(\overline{M}^{n+1}, \overline{g})$ with tangential part $V$. Then, the following identity $\mathcal{L}_{\overline{V}} \overline{g} - \mathcal{L}_{V} g = - 2 \epsilon_N C \mathcal{A}$  holds on $M^n$.
\end{lemma}

\begin{proof}
Take $Y, Z \in \mathfrak{X}(M)$ and write $\overline{V} = V + \epsilon_N\overline{g}(\overline{V},N)N$. Therefore, we have on $M$
\begin{equation*}
\overline{\nabla}_Y \overline{V} = \nabla_Y V + \alpha (Y,V) + \epsilon_N Y \overline{g}(\overline{V},N)N - \epsilon_N \overline{g}(\overline{V},N) \mathcal{A} Y.
\end{equation*}
Notice that this latter equation remains valid if we swap $Y$ with $Z$, and therefore we deduce 
\begin{eqnarray*}
(\mathcal{L}_{\overline{V}} \overline{g})(Y,Z) &=&  \overline{g}(\overline{\nabla}_Y \overline{V},Z) + \overline{g}(\overline{\nabla}_Z \overline{V},Y) \\
&=& (\mathcal{L}_{V} g) (Y,Z) - 2 \epsilon_N \overline{g}(\overline{V},N)g(\mathcal{A}Y,Z).
\end{eqnarray*}
\end{proof}
Clearly, this lemma is valid for an arbitrary vector field on the ambient manifold. In our context, however, Lemma \ref{lemmaDerivadaLie} measures to what extent the tangent component  $V$ of the conformal vector field $\overline{V}$ deviates from being itself conformal. The second lemma is an immediate consequence of the former and determines the necessary and sufficient conditions to immerse isometrically a Ricci almost soliton into a semi-Riemannian ambient manifold.

%================================================================================================
%                                                                               LEMMA 2
%================================================================================================

\begin{lemma}\label{1}
Let $(M^n,g)$ be a Riemannian manifold isometrically immersed into a semi-Riemannian manifold  $(\overline{M}^{n+1}, \overline{g})$ endowed with a conformal vector field $\overline{V}$ with conformal factor $\sigma$. Then, $V$ brings forth  a Ricci almost soliton structure on $M^n$ with soliton function $\lambda$ if and only if $Ric=-\psi g- \epsilon_N C \mathcal{A}$, where $\psi=\sigma|_M-\lambda$.
\end{lemma}

\begin{proof}
	 By hypothesis we have  that $\mathcal{L}_{\overline{V}}\overline{g} = 2 \sigma \overline{g}$. Then, Lemma~\ref{lemmaDerivadaLie} implies
	\begin{equation}\label{eq11}
	\mathcal{L}_Vg= 2 \sigma \overline{g}+ 2 \epsilon_NC \mathcal{A},
	\end{equation}
	where $\mathcal{A}$ is the symmetric $(0,2)$-tensor associated to the  $(1,1)$-tensor $\mathcal{A}$. If $V$ defines the Ricci almost soliton structure on $M^n$, then $Ric + \frac{1}{2}\mathcal{L}_Vg = \lambda g$. Substituting \eqref{eq11} in the latter equation we obtain
	\begin{equation*}
	Ric=-\psi g - \epsilon_N C \mathcal{A},
	\end{equation*}
	where $\lambda$ is the soliton function and $\psi=\sigma|_M-\lambda$. Conversely, supposing that $Ric = - \psi g - \epsilon_N C \mathcal{A}$, then
	\begin{equation*}
	Ric = \lambda g - \frac{1}{2}\mathcal{L}_{\overline{V}}\overline{g} - \epsilon_N C \mathcal{A},
	\end{equation*}
	which, again by Lemma~\ref{lemmaDerivadaLie}, clearly reduces to  $Ric + \frac{1}{2}\mathcal{L}_Vg = \lambda g$.
\end{proof}
We could not emphasise more the fact that, albeit being a simple result at first sight, this latter lemma lies at the core of this article. Verily, its generality is readily perceived in the following remark.
\begin{remark}\label{r1}
The function $\psi$ plays an important role in this article and henceforth we shall always bear in mind that $\psi=\sigma|_M-\lambda$. Notice that in the context of Lemma~\ref{1} there is no restriction on the function $\psi$. For instance, if $M^n$ is Ricci flat and totally geodesic, then $\psi\equiv 0$. Indeed, this is exactly the case with $\mathbb{R}^n$ as we shall see in Example~\ref{exem1}. In the case when $M^n$ is Ricci flat and totally umbilical then 
$\psi= -CH$. In other words, Lemma~\ref{1} is  as general as possible, and under its hypotheses, it captures all the Ricci almost soliton structures on $M^n$ which arise from the tangent part of  a conformal vector field on the ambient space. However, in Section~\ref{sec4} we shall necessarily impose, in the context of our Theorem~\ref{thm2}, that $\psi\neq 0$ and therefore we shall naturally exclude the flat case from the corresponding considerations therein (see also Remark~\ref{r2}).
\end{remark}
 We now continue with  two technical lemmas. Notice that the first one is auxiliary in character, and it is only to be exploited in  the proof of the second.  %================================================================================================
%                                                                               LEMMA 3
%================================================================================================

\begin{lemma}\label{gradC}
Let $(M^n,g)$ be a Riemannian manifold isometrically immersed into a semi-Riemannian manifold $(\overline{M}^{n+1}, \overline{g})$ endowed with a conformal vector field $\overline{V}\in\mathfrak{X}(\overline{M}^{n+1})$ with conformal factor $\sigma$. Then, the gradient of the angle function $C$ is given by the expression  $\nabla C = - (\overline{\nabla}_N\overline{V})^T - \mathcal{A}V$.
\end{lemma}

\begin{proof}
By definition of the gradient we have for $X \in \mathfrak{X}(M^n)$
\begin{eqnarray*}
X C &=&  \overline{g}(\overline{\nabla}_X \overline{V},N) + \overline{g}(\overline{V},\overline{\nabla}_XN) \\
&=& (\mathcal{L}_{\overline{V}}\overline{g})(X,N) -  \overline{g}(\overline{\nabla}_{N} \overline{V},X) - \overline{g}(\overline{V},\mathcal{A}X) \\
&=& 2 \sigma \overline{g}(X,N) - \overline{g}((\overline{\nabla}_{N} \overline{V})^T,X) - g(V,\mathcal{A}X) \\
&=& - g((\overline{\nabla}_{N} \overline{V})^T,X) - g(\mathcal{A}V,X).
\end{eqnarray*}
\end{proof}

%================================================================================================
%                                                                               LEMMA 4
%================================================================================================

Before we state and prove our last lemma, we should like to remark that it is precisely this lemma that will manifest itself in the fact that the angle function $C$ is a special scalar concircular field.

\begin{lemma}\label{HessC2}
Let $(M^n,g)$ be a  Riemannian manifold isometrically immersed in a  semi-Riemannian manifold $(\overline{M}^{n+1}_{c}, \overline{g})$ endowed with a conformal vector field $\overline{V}\in\mathfrak{X}(\overline{M})$ with conformal factor $\sigma$. Then, the Hessian of the angle function $C$ is given by the following expression:
\begin{eqnarray*}
\mathrm{Hess}\,C(X,Y) &=&  -(cC + N \sigma) g(X,Y) + \sigma g(\mathcal{A}X,Y) + \epsilon_N C g(\mathcal{A}^2X,Y) \\
&& - g((\nabla_{V}\mathcal{A})X,Y) - g(\mathcal{A}\nabla_X V,Y) - g(\mathcal{A} \nabla_Y V,X).
\end{eqnarray*}
\end{lemma}

\begin{proof}
By dint of Lemma~\ref{gradC}, we readily compute
\begin{align*}
\mathrm{Hess}\, C(X,Y) &= g(\nabla_X \nabla C , Y) = - \overline{g}(\overline{\nabla}_X  (\overline{\nabla}_N\overline{V})^T,Y) - \overline{g}(\nabla_X\mathcal{A}V,Y) \\
&= - \overline{g}(\overline{\nabla}_X  \overline{\nabla}_N\overline{V},Y) + \overline{g} (\overline{\nabla}_X (\overline{\nabla}_N\overline{V})^{\perp},Y)- g(\nabla_X\mathcal{A}V,Y) \\
&=- \overline{g}(\overline{\nabla}_X  \overline{\nabla}_N\overline{V},Y) + \overline{g} (\overline{\nabla}_X (\overline{\nabla}_N\overline{V})^{\perp},Y) - g((\nabla_{X}\mathcal{A})V,Y) - g(\mathcal{A}\nabla_X V,Y).
\end{align*}
Note that
\begin{equation*}
(\overline{\nabla}_N\overline{V})^{\perp} = \epsilon_N \overline{g}(\overline{\nabla}_N\overline{V},N)N = \frac{\epsilon_N}{2}(\mathcal{L}_{\overline{V}}\overline{g})(N,N)N = \epsilon_N(\sigma\overline{g}(N,N))N = \sigma N.
\end{equation*}
Using this latter equality and the fact that $\mathcal{A}$ is a Codazzi tensor we have
\begin{align*}
\mathrm{Hess}\, C(X,Y) &= - \overline{g}(\overline{\nabla}_X  \overline{\nabla}_N\overline{V},Y) + \overline{g}(\overline{\nabla}_X(\sigma N),Y) - g((\nabla_{V}\mathcal{A})X,Y) - g(\mathcal{A}\nabla_X V,Y) \\
&=- \overline{g}(\overline{\nabla}_X  \overline{\nabla}_N\overline{V},Y) + \sigma \overline{g}(\overline{\nabla}_XN,Y) - g((\nabla_{V}\mathcal{A})X,Y) - g(\mathcal{A}\nabla_X V,Y) \\
&= - \overline{g}(\overline{\nabla}_X  \overline{\nabla}_N\overline{V},Y) -\sigma g(\mathcal{A}X,Y) - g((\nabla_{V}\mathcal{A})X,Y) - g(\mathcal{A}\nabla_X V,Y).
\end{align*}
Consider now the local extensions $\overline{X}$ and $\overline{Y}$ of $X$ and $Y$, respectively, and such that
\begin{equation}\label{extloc}
\overline{\nabla}_N \overline{X} = \overline{\nabla}_N \overline{Y} = 0.
\end{equation}
Thus,
\begin{eqnarray*}
- \overline{g}(\overline{\nabla}_X  \overline{\nabla}_N\overline{V},Y) &=& \overline{g}(\overline{R}(\overline{X},N)\overline{V},\overline{Y}) - \overline{g}(\overline{\nabla}_N  \overline{\nabla}_{\overline{X}}\overline{V},\overline{Y}) - \overline{g}(\overline{\nabla}_{[\overline{X},N]}\overline{V},\overline{Y}).
\end{eqnarray*}
The fact that $\overline{M}^{n+1}$ has constant sectional curvature $c$ together with  \eqref{extloc} imply
\begin{align*}
- \overline{g}(\overline{\nabla}_X  \overline{\nabla}_N\overline{V},Y) &= -cC \overline{g}(\overline{X},\overline{Y}) - \overline{g}(\overline{\nabla}_N  \overline{\nabla}_{\overline{X}}\overline{V},\overline{Y}) - \overline{g}(\overline{\nabla}_{\overline{\nabla}_{\overline{X}}N}\overline{V},\overline{Y}) + \overline{g}(\overline{\nabla}_{\overline{\nabla}_N\overline{X}}\overline{V},\overline{Y})\\
&= -cC g(X,Y) - N \overline{g}( \overline{\nabla}_{\overline{X}}\overline{V},\overline{Y}) - \overline{g}(\overline{\nabla}_{\overline{X}}\overline{V},\overline{\nabla}_N\overline{Y}) - \overline{g}(\overline{\nabla}_{\overline{\nabla}_XN}\overline{V},\overline{Y}) \\
&= -cC g(X,Y) - N \overline{g}( \overline{\nabla}_{\overline{X}}\overline{V},\overline{Y}) + g(\overline{\nabla}_{\mathcal{A}X}\overline{V},Y).
\end{align*}
Therefore,
\begin{eqnarray*}
\mathrm{Hess}\, C (X,Y) &=&  -cC g(X,Y) - N \overline{g}( \overline{\nabla}_{\overline{X}}\overline{V},\overline{Y}) + \overline{g}(\overline{\nabla}_{\mathcal{A}X}\overline{V},Y) - \sigma g (\mathcal{A}X,Y)\\
&&- g((\nabla_{V}\mathcal{A})X,Y) - g(\mathcal{A}\nabla_X V,Y).
\end{eqnarray*}
Moreover,
\begin{eqnarray*}
\overline{\nabla}_{\mathcal{A}X}\overline{V} &=& \overline{\nabla}_{\mathcal{A}X} (V + \epsilon_N CN) \\
&=& \overline{\nabla}_{\mathcal{A}X} V + \epsilon_N C \overline{\nabla}_{\mathcal{A}X} N  + \epsilon_N \mathcal{A}X(C)N \\
&=& \overline{\nabla}_{\mathcal{A}X} V - \epsilon_N C \mathcal{A}^2X - \epsilon_N \mathcal{A}X(C)N.
\end{eqnarray*}
Taking the tangential projection we get $(\overline{\nabla}_{\mathcal{A}X} \overline{V})^T = \nabla_{\mathcal{A}X} V  - \epsilon_N C \mathcal{A}^2X$. Thereby,
\begin{equation*}
\overline{g}(\overline{\nabla}_{\mathcal{A}X} \overline{V},Y) = \overline{g}((\overline{\nabla}_{\mathcal{A}X} \overline{V})^T,Y) = g(\nabla_{\mathcal{A}X} V,Y) - \epsilon_N C g(\mathcal{A}^2X,Y).
\end{equation*}
Thus,
\begin{eqnarray*}
\mathrm{Hess}\, C (X,Y) &=&  -cC g(X,Y) -N \overline{g}( \overline{\nabla}_{\overline{X}}\overline{V},\overline{Y}) + g(\nabla_{\mathcal{A}X} V,Y) - \epsilon_N C g(\mathcal{A}^2X,Y)\\
&& - \sigma g(\mathcal{A}X,Y)- g((\nabla_{V}\mathcal{A})X,Y) - g(\mathcal{A}\nabla_X V,Y).
\end{eqnarray*}
Now, write
$$\mathrm{Hess}\, C (X,Y) = \mathcal{S}(X,Y)+\mathcal{T}(X,Y),$$
where
\begin{equation}\label{eqS}
\mathcal{S}(X,Y)=-cC g(X,Y) - \epsilon_N C g(\mathcal{A}^2X,Y) - \sigma g(\mathcal{A}X,Y) - g((\nabla_{V}\mathcal{A})X,Y)
\end{equation}
and
\begin{equation*}
\mathcal{T}(X,Y)=-N \overline{g}( \overline{\nabla}_{\overline{X}}\overline{V},\overline{Y}) + g(\nabla_{\mathcal{A}X} V,Y) - g(\mathcal{A}\nabla_X V,Y).
\end{equation*}
Clearly, $\mathcal{T}$ is a symmetric tensor as both $\mathrm{Hess}\, C$ and $\mathcal{S}$ are. So,
\begin{eqnarray*}
2 \mathcal{T}(X,Y) &=& \mathcal{T}(X,Y) + \mathcal{T}(Y,X) \\
&=& - N [\overline{g}( \overline{\nabla}_{\overline{X}}\overline{V},\overline{Y}) + \overline{g}( \overline{\nabla}_{\overline{Y}}\overline{V},{\overline{X}})] + g(\nabla_{\mathcal{A}X} V,Y) + g(\nabla_{\mathcal{A}Y} V,X) \\
&&- g(\mathcal{A}\nabla_X V,Y) - g(\mathcal{A}\nabla_Y V,X) \\
&=& - N [(\mathcal{L}_{\overline{V}} \overline{g})(\overline{X},\overline{Y})] + (\mathcal{L}_Vg)(\mathcal{A}X,Y)-g(\nabla_Y V,\mathcal{A}X)\\
&&+ (\mathcal{L}_Vg)(\mathcal{A}Y,X)-g(\nabla_X V,\mathcal{A}Y) - g(\mathcal{A}\nabla_X V,Y) - g(\mathcal{A}\nabla_Y V,X).
\end{eqnarray*}
Using the fact that the vector field $\overline{V}$ is conformal we have 
\begin{eqnarray*}
2 \mathcal{T}(X,Y) &=& - 2N[\sigma \overline{g}(\overline{X},\overline{Y})] + (\mathcal{L}_Vg)(\mathcal{A}X,Y) + (\mathcal{L}_Vg)(\mathcal{A}Y,X) \\
&& - 2g(\mathcal{A}\nabla_X V,Y) - 2g(\mathcal{A}\nabla_Y V,X) \\
&=& -2N(\sigma) \overline{g}(\overline{X},\overline{Y}) - 2 \sigma N[\overline{g}(\overline{X},\overline{Y})] \\
&&+ (\mathcal{L}_Vg)(\mathcal{A}X,Y) + (\mathcal{L}_Vg)(\mathcal{A}Y,X) - 2g(\mathcal{A}\nabla_X V,Y) - 2g(\mathcal{A}\nabla_Y V,X)\\
&=& -2N(\sigma) \overline{g}(\overline{X},\overline{Y}) - 2 \sigma [\overline{g}(\overline{\nabla}_N \overline{X},\overline{Y}) + \overline{g}(\overline{X},\overline{\nabla}_N\overline{Y})] \\
&&+ (\mathcal{L}_Vg)(\mathcal{A}X,Y) + (\mathcal{L}_Vg)(\mathcal{A}Y,X) - 2g(\mathcal{A}\nabla_X V,Y) - 2g(\mathcal{A}\nabla_Y V,X).
\end{eqnarray*}
Now, it follows from equation~\eqref{extloc} and Lemma~\ref{lemmaDerivadaLie} that
\begin{eqnarray*}
2 \mathcal{T}(X,Y) &=& -2N(\sigma) g(X,Y) + 2 \sigma g(\mathcal{A}X,Y) + 2 \epsilon_N C g(\mathcal{A}^2X,Y) \\
&&+  2 \sigma g(\mathcal{A}X,Y) + 2 \epsilon_N C g(\mathcal{A}^2X,Y) - 2g(\mathcal{A}\nabla_X V,Y) - 2g(\mathcal{A}\nabla_Y V,X) \\
&=& -2N(\sigma) g(X,Y) + 4 \sigma g(\mathcal{A}X,Y) + 4 \epsilon_N C g(\mathcal{A}^2X,Y) \\
&&- 2g(\mathcal{A}\nabla_X V,Y) - 2g(\mathcal{A}\nabla_Y V,X).
\end{eqnarray*}
Thus,
\begin{eqnarray}\label{eqT}
\mathcal{T}(X,Y) &=& -N(\sigma) g(X,Y) + 2 \sigma g(\mathcal{A}X,Y) + 2 \epsilon_N C g(\mathcal{A}^2X,Y) \\
\nonumber &&- g(\mathcal{A}\nabla_X V,Y) - g(\mathcal{A}\nabla_Y V,X).
\end{eqnarray}
Summing up equations ~\eqref{eqS} and ~\eqref{eqT} concludes the proof of the lemma.
\end{proof}

%=================================================================================================
%              CONFORMAL VECTOR FIELDS  WHICH GIVE RISE TO RICCI ALMOST SOLITON STRUCTURE
%=================================================================================================

\section{Conformal vector fields which give rise to Ricci almost soliton structure}\label{sec4}
In this section, we shall prove two theorems that reveal the intimate relationship between the Ricci almost soliton structure and conformal vector fields. These theorems will be thoroughly exploited in Section \ref{sec5} in the construction of concrete examples and will culminate in the classification theorem discussed in Section \ref{conclusions}. We begin with a brief discussion which prepares the ground for what is to follow.
\begin{proposition}\label{prop1}
Let $(M^n,g)$ be a totally umbilic Riemannian manifold isometrically immersed into a  semi-Riemannian manifold $(\overline{M}^{n+1}_{c}, \overline{g})$ endowed with a conformal vector field $\overline{V}\in\mathfrak{X}(\overline{M})$ with conformal factor $\sigma$. Then, the angle function $C$ is a special scalar concircular field on $M^n$.
\end{proposition}
\begin{proof}
 To prove that the angle function $C$ is a special scalar concircular field  on $M^n$ is to prove that it satisfies the equation
\begin{equation}\label{HessUmb}
\mathrm{Hess}\, C=(-kC + b)g,
\end{equation}
for some constants $k,\,b\in\mathbb{R}$. Since $M^n$ is totally umbilic we have that the mean curvature $H$ is constant and $\mathcal{A}X=\epsilon_N HX$. We readily perceive the following two implications of the latter hypothesis. Firstly, it follows from Lemma~\ref{HessC2} that 
\begin{eqnarray*}
\mathrm{Hess}\, C (X,Y) &=& -(cC + N \sigma)g(X,Y)  +\epsilon_N \sigma H g(X,Y) + \epsilon_N C H^2g(X,Y) \\
&& - \epsilon_N H g(\nabla_X V,Y) - \epsilon_N H g(\nabla_Y V,X) \\
&=& -[cC + N\sigma - \epsilon_N \sigma H - \epsilon_N C H^2] g(X,Y) - \epsilon_N H (\mathcal{L}_{V}g)(X,Y)
\end{eqnarray*}
Now, by applying Lemma~\ref{lemmaDerivadaLie} we compute further that
\begin{eqnarray*}
\mathrm{Hess}\, C (X,Y) &=& -[cC + N\sigma - \epsilon_N \sigma H - \epsilon_N C H^2] g(X,Y) - 2\epsilon_N H[\sigma+ CH] g(X,Y)\\
&=&  -[(c+\epsilon_N H^2)C + N\sigma + \epsilon_N \sigma H] g(X,Y).
\end{eqnarray*}
Wherefore, the angle function $C$ indeed satisfies equation~\eqref{HessUmb} for $k=(c+\epsilon_N H^2)$ and $b=-\Big(N\sigma + \epsilon_N \sigma H\Big)$. By hypothesis, $k$ is constant. Notice that, by taking trace of equation~\eqref{EqRicM}, we obtain for the scalar curvature  $S$ of  $M^n$ to be given by $\dfrac{S}{n(n-1)}=c+\epsilon_N H^2=k$, just as to be expected. It only remains to show that  the summand   $N\sigma + \epsilon_N \sigma H$ is also constant on $M^n$.  A well-known fact is that the existence of a conformal vector field $\overline{V}$ on $\overline{M}^{n+1}_{c}$ implies
\begin{equation}\label{eqHesssigma}
 \overline{\mathrm{Hess}}\,\sigma = - c \sigma \overline{g}. 
 \end{equation}
 The validity of the latter equation can be consulted in  Ishihara-Tashiro~\cite{Ishihara-Tashiro} or Yano~\cite{yano, yano2}. Now, taking an arbitrary vector field $X\in \mathfrak{X}(M^n)$, considering the fact that $M^n$ is totally umbilic, and making use of the Weingarten formula as well as equation~\eqref{eqHesssigma}, one readily computes that 
 \begin{align*}
 X\Big(N\sigma + \epsilon_N \sigma H\Big)&=XN\sigma+\epsilon_N\sigma HX\sigma \\
 &= \overline{\mathrm{Hess}}\sigma(X,N)+\Big(\overline\nabla_XN\Big)\sigma + \epsilon HX\sigma \\
 &= -c\sigma\overline{g}(X,N)-\mathcal{A}X\sigma+\epsilon_NHX\sigma \\
 &=0,
 \end{align*}
  which means that $b$ is constant on $M^n$ and the proof is complete.
  \end{proof} 
 Notice that this proposition, and equation~\eqref{HessUmb} in particular, lie at the heart of proof of both theorems below. We have just arrived at our first theorem, which in essence, is a manifestation of Theorem A into the realm of Ricci almost solitons.

%==============================================================================================
 %                                                                THEOREM 1
%==============================================================================================

\begin{theorem}\label{thm1}
Let  $(M^n,g)$ be a connected, totally umbilic, Riemannian hypersurface of a semi-Riemannian manifold $(\overline{M}^{n+1}_{c}, \overline{g})$. Then, the Ricci almost soliton structure on  $M^n$ is determined by the tangent component of a conformal vector field $\overline{V}$ on $\overline{M}^{n+1}_{c}$. Furthermore, for any vector field $V$ which determines a Ricci almost soliton structure on $M^n$ there exists a unique conformal vector field  $\overline{V}$ on $\overline{M}^{n+1}_{c}$ so that the tangential part of  $\overline{V}|_{M^n}$ is $V$.
\end{theorem}
\begin{proof}
For the reader familiar with the proof of Theorem A, the following simple observation would suffice. Recall that in our context the Gauss equation is naturally written in the form
$$Ric=c(n-1)g +nH\mathcal{A}-\epsilon_N\mathcal{A}^2.$$
By virtue of this equation, it is readily understood that totally umbilic hypersurfaces of semi-Riemannian manifold of constant sectional curvature are necessarily Einstein manifolds. Now, it only takes to recall the definitions of conformal vector field and Ricci almost soliton to contemplate the truth of the following affirmation. Any vector field that determines the Ricci almost soliton structure on $M^n$ is a conformal vector field on $M^n$, and vice versa. This latter affirmation, together with the proof of Theorem A, virtually proves our theorem. 

To facilitate the reader not familiar with Theorem A and its proof, however, we should like to proceed with the details. As our exposition reminisces the original proof of Kim et al., the reader may also wish to consult \cite{Kim}. In what follows, we write $\mathcal{C}(M)$ and $\mathcal{C}(\overline{M})$ for the vector spaces of conformal vector fields on $M^n$ and $\overline{M}^{n+1}_{c}$, respectively. The statement of our theorem clearly suggests that we need to find a natural way to compare these latter spaces. To this end, we define the linear map $\Psi: \mathcal{C}(\overline{M}) \to \mathcal{C}(M)$ by $\Psi(\overline{V})=V$, where $V$ is the tangent component of $\overline{V}$ on $M^n$. It ought to be evident at this juncture that it is our first duty to prove the surjectivity of $\Psi$. As $M^n$ is umbilic, it is guaranteed by Lemma~\ref{lemmaDerivadaLie} that $\Psi$ is well defined. Moreover,  as $M^n$ is totally umbilic it follows from Proposition~\ref{prop1} that $\mathrm{Hess}\, C=(-kC + b)g$ holds. Write $\mathcal{FC}(M^n)$ for the space of all smooth functions on $M^n$ satisfying this latter equation, namely, the space of all special scalar conscircular fields.
Now, observe that, if $\overline{V} \in Ker(\Psi)$, then $\overline{V}=\epsilon_NCN$. This way, we define the linear map $\Phi: Ker(\Psi) \to \mathcal{FC}(M^n)$ by $\Phi(\overline{V})=\epsilon_NC$. It is easy to see, that $\Phi$ is injective. Indeed, if $\overline{V} \in  Ker(\Phi)$, then $C=0$, which is, $\overline{V}^{\perp}=0$, and on the other hand, $\overline{V} \in Ker(\Psi)$, implies $\overline{V}=0$. It is known that $\mathrm{dim}\,\mathcal{FC}(M^n)=n+2$ (see Tashiro-Miyashita~\cite{Tashiro-Miyashita}). Thence, by applying the Rank - Nullity Theorem to the map $\Phi$, the verity of the following inequality is  confirmed
\begin{equation}\label{des1}
\mathrm{dim}\, Ker(\Psi) \leq \mathrm{dim}\, \mathcal{FC}(M^n) = n+2.
\end{equation}
It is also known that $\mathrm{dim}\, \mathcal{C}(\overline{M})=\dfrac{(n+2)(n+3)}{2}$ and $\mathrm{dim}\, \mathcal{C}(M) = \dfrac{(n+1)(n+2)}{2}$ (see Kobayashi~\cite{Kobayashi}). So, applying the Rank - Nullity Theorem  to the map $\Psi$ brings forth the inequality 
 \begin{equation}\label{des2}
 n+2\leq \mathrm{dim}\, Ker(\Psi).
 \end{equation}
 Ergo, \eqref{des1} and \eqref{des2} imply that $\mathrm{dim}\, Ker(\Psi)=\mathrm{dim}\, \mathcal{FC}(M^n)$. To put it differently, $\Phi$ is surjective and therefore bijective. This latter bijection, in turn, implies that $\Psi$ is surjective. It is this surjection that guarantees the truth of the following affirmation. For any conformal vector field $V$ on $M^n$ there exists a  conformal vector field $\overline{V}$ on $\overline{M}^{n+1}_{c}$ such that $\overline{V}^T=V$. Conversely, for any fixed  conformal vector field $V$ on $M^n$ we can choose a  conformal vector field $\overline{V}_0$ on $\overline{M}^{n+1}_{c}$ such that $\Psi(\overline{V}_0)=\overline{V}_0^T=V$. Write  $\overline{V}_1 = \epsilon_NCN$ for the orthogonal complement of $\overline{V}_0$, and recall  that $\overline{V}_1 \in Ker(\Psi)$. Since $C, -C \in \mathcal{FC}(M^n)$ and $\Phi$ is a bijection, the uniqueness of $\overline{V}_1$ is perceived. In summary, starting from a conformal vector field $V$ on $M^n$, we constructed a unique conformal vector field $\overline{V}$ on $\overline{M}^{n+1}_{c}$, satisfying $\overline{V} = \overline{V}_0+\overline{V}_1$ and such that $\overline{V}|_M=V$. With this, the proof of the theorem is complete.
\end{proof}
The following remark is now due.
 \begin{remark}\label{r2}
As we shall see soon in Section~\ref{sec5}, Theorem~\ref{thm1}, Lemma~\ref{1} and equation~\eqref{campoconfRn+1} allow us to characterise all structures of Ricci almost solitons on the space forms $\mathbb{R}^n,\, \mathbb{S}^n$, and $\mathbb{H}^n$. It will turn out that in the Euclidean case $\psi\equiv 0$, while in the other two cases $\psi\neq 0$ as well as the Ricci tensor is multiple of the Weingarten operator $\mathcal{A}$. Notice that the structure of  Ricci almost soliton on the Euclidean sphere has already been characterised by the first author in his thesis~\cite{NazaPhD} by means of the first eigenvalue of the Laplacian operator on the Euclidean sphere. Notwithstanding, a characterisation for neither $\mathbb{R}^n$ nor $\mathbb{H}^n$ that uses the latter technique remains unknown.
\end{remark}

%==============================================================================================
 %                                                                THEOREM 2
%==============================================================================================

\begin{theorem}\label{thm2}
Let $(M^n,g)$ be a complete Riemannian manifold isometrically immersed into a  semi-Riemannian manifold $(\overline{M}^{n+1}_{c}, \overline{g})$. Let  $\overline{V}$ be a conformal vector field with conformal factor  $\sigma$  whose tangent component defines a Ricci almost soliton structure on $M^n$. Let  $\lambda$ be the soliton function, and suppose that $\psi=\sigma|_{M^n}-\lambda$ is non-zero in some dense subset of $M^n$ and $Ric = \mu \mathcal{A}$, for some smooth function $\mu$ on $M^n$. Then, $C$ is a special scalar concircular field on  $(M^n,g)$.
\end{theorem}

\begin{proof}
By dint of  Lemma \ref{1} we write
\begin{equation}\label{pshimuC}
\mu \mathcal{A} = - \psi I - \epsilon_N C \mathcal{A},
\end{equation}
whence
$$\mathcal{A} = - \frac{\psi}{\mu +  \epsilon_N C} I,$$
for all $p \in M^n$ such that $\mu(p) +  \epsilon_N C(p) \neq 0$. In other words, the open set
$$\mathcal{U} = \left\lbrace p \in M^n : \mu(p) +  \epsilon_N C(p) \neq 0 \right\rbrace$$
of $M^n$ is totally umbilic on $\overline{M}_{c}^{n+1}$ with constant mean curvature $H=- \dfrac{\epsilon_N\psi}{\mu+\epsilon_N C}$. In this case, as already seen in  equation~\eqref{HessUmb}, it is valid on $\mathcal{U}$ that
\begin{equation*}
\mathrm{Hess}\, C=(-kC + b)g,
\end{equation*}
where $k=\dfrac{S}{n(n-1)}$ and $b=-N\sigma - \epsilon_N \sigma H$ are constants. Therefore, $C$ is a special scalar concircular field on $\mathcal{U}$. If $\mathcal{U}$ is dense in $M$, then, by continuity, we have that $\mathrm{Hess}\, C=(-kC + b)g$ holds on the whole of $M$. The following  remark is worth noting.  If  $\mathcal{U}$ is not dense  in $M$, then there must exis an open  neighbourhood $\mathcal{W}$ on $M-\mathcal{U}$ such that $\mu +\epsilon_NC =0$ on $\mathcal{W}$. This latter assumption, along with \eqref{pshimuC}, will imply that $\psi$ is identically zero on $\mathcal{W}$. This contradiction completes the proof of the theorem.
\end{proof}
 The following remark is worth noting. In our context the class of umbilical hypersurfaces is included in the class of surfaces satisfying $Ric=\mu\mathcal{A}$. Indeed, this is guaranteed by Gauss equation~\eqref{EqRicM}.
%=================================================================================================
%                                                             EXAMPLES / APPLICATIONS
%=================================================================================================

\section{Potential fields on some totally umbillic hypersurfaces in Euclidean and Lorentz spaces}\label{sec5}

Recall that a potential field is just a gradient vector field defining a Ricci almost structure, or equivalently, a vector field defining a gradient Ricci almost soliton structure. In this section we construct explicit examples of gradient Ricci almost solitons immersed in  Euclidean and Lorentz spaces. These examples not only illustrate the power of Theorem~\ref{thm1} but also allow us to relate our work to some already published papers, and in particular to the main results of \cite{Pacific}, see the concluding section. In other words, the general nature of Theorem~\ref{thm1} is justified herein. It is important to emphasise that, Examples~\ref{exem1}~and~\ref{exem2} can be constructed either by Theorem~\ref{thm1} or previously known techniques. Nonetheless, the only method  to construct Examples~\ref{exem3}~and~\ref{exem4}, to the best of our knowledge, is due to Theorem~\ref{thm1}. In summary, the goal of this section is twofold. On the one hand, it shows that Theorem~\ref{thm1} is not vacuous. On the other, it provides some crucial insights which lead to the proof of our classification theorem.

In \cite{Kim} is given a complete description of conformal vector fields on the semi-Euclidean space $\mathbb{R}_{\nu}^{n+1}$. In order to prepare the ground for the examples to follow, we must begin with a brief sketch of this matter. Consider the semi-Euclidean space $(\Bbb{R}_{\nu}^{n+1},\overline{g})$ with metric given by $d s^2 = \displaystyle\sum_{i=1}^{n+1}\epsilon_i dx_i^2$, where
$\epsilon_1=\cdots=\epsilon_{\nu}=-1$, $\epsilon_{\nu +1}=\cdots = \epsilon_{n+1}=1$. In this way one indeed captures semi-Riemannian metrics of arbitrary signature, and  the particular cases of Euclidean and Lorentzian spaces correspond to $\nu=0$ and $\nu=1$, respectively. We shall write $x$ for the point $(x_1,\ldots,x_{n+1})\in \Bbb{R}_{\nu}^{n+1}$ and $\vec{x}$ for the position vector in $\Bbb{R}_{\nu}^{n+1}$ with coordinates $(x_1,\ldots,x_{n+1})$. Let $\overline{V}$ be a conformal vector field on $\mathbb{R}_{\nu}^{n+1}$. Then, since our ambient space $\Bbb{R}_{\nu}^{n+1}$ is flat it is readily perceived that  equation \eqref{eqHesssigma} reduces to $\overline{\mathrm{Hess}}\,\sigma = 0$, and by integration one finds that for all $\vec{x} = (x_1,\ldots,x_{n+1}) \in \Bbb{R}_{\nu}^{n+1}$ the solution is given by
\begin{equation}\label{sigma}
\sigma(x)= \overline{g}(\vec{a},\vec{x})  + \beta,
\end{equation}
where $\vec{a} = (\epsilon_1a_1,\ldots,\epsilon_{n+1}a_{n+1})$ is some constant vector and $\beta$ is a real number. It is perhaps worth noticing swiftly here  that $ \overline{g}(\vec{a},\vec{x})= h_{\vec{a}}(x)$, where $ h_{\vec{a}}(x)$ is the height function. Now, starting from the equation $\mathcal{L}_{\overline{V}}\overline{g}=2 \sigma \overline{g}$ and performing some routine calculations one gets the following description of the conformal vector fields on $\Bbb{R}_{\nu}^{n+1}$
\begin{equation}\label{campoconfRn+1}
\overline{V}(x)= \sigma (x)\vec{x} - \frac{1}{2}\overline{g} (\vec{x}, \vec{x}) \vec{a} + B \vec{x} + \frac{1}{2}\vec{\gamma},
\end{equation}
where $\vec{\gamma} = (\gamma_1,\ldots,\gamma_{n+1})$ is some constant vector and $B$ is $(n+1)\times(n+1)$ matrix with entries satisfying $\epsilon_jb_{jk}+\epsilon_kb_{kj}=0$ for distinct $j,k\in\{1,\ldots, n+1\}$, and $b_{ii}=0$. Notice that in the Euclidean case the matrix $B$ will be an anti-symmetric matrix. It must also be noticed that $B \vec{x} + \frac{1}{2}\vec{\gamma}$ is the Killing part of the $\overline{V}$ in $\mathbb{R}_{\nu}^{n+1}$, see~\cite{O'Neill}.

Having completed our preliminary discussion, we are now ready to present our examples. For brevity, we shall only present the key points in each construction. All the examples herein can be found elaborated in detail in the thesis of the second author \cite{pereira}.  We begin with the flat case, being the simplest possible. 

%================================================================================================
%                                                                  EXAMPLE 1
%================================================================================================

\begin{exem}\label{exem1}
Consider $(\Bbb{R}^n,g)$ as a hypersurface in the Euclidean space $(\Bbb{R}^{n+1},\overline{g})$.  Theorem~\ref{thm1} immediately  implies that the tangent component of the conformal vector field  \eqref{campoconfRn+1} determines a Ricci almost soliton structure on $\mathbb{R}^n$. Notice that in this case we just have $V=\overline{V}\vert_{\mathbb{R}^n}$. As the Ricci tensor of $\Bbb{R}^n$ is identically zero the following equation is valid
\begin{equation*}
\mathcal{L}_V g = 2 \lambda g,
\end{equation*}	
where $\lambda$ is the soliton function of $\Bbb{R}^n$. Observe that we are in the case of a totally geodesic hypersurface, which is,  the Weingarten operator $\mathcal{A}$ is identically zero. Thus, it follows from Lemma~\ref{1} that $\sigma \vert_{\Bbb{R}^n} = \lambda$, where $\sigma$ is given by equation~\eqref{sigma}.
\end{exem}
\begin{remark}The intimate relationship between the Ricci almost soliton structure and the existence of conformal fields is indeed evident from this latter example. It is in fact immediately perceived that the existence of Ricci almost soliton structure on $\mathbb{R}^n$ necessitates the existence of a conformal vector field on $\mathbb{R}^{n+1}$and vice versa. It is only in this case, however, that the soliton function $\lambda$ coincides with the restriction of the conformal function $\sigma\vert_{M^n}$. In the examples to follow, we shall see that the soliton function is more complicated.\end{remark}
%================================================================================================
%                                                                  EXAMPLE 2
%================================================================================================

 We now need to prepare the ground before we construct the remaining examples. The following brief discussion enables us to kill two birds with one stone. We write

\begin{equation*}
 \mathbb{M}^n(c) =
    \begin{cases}
      \mathbb{S}^n & \mathrm {for}\,\,\, c=1, \\
       \mathbb{H}^n &  \mathrm {for}\,\,\, c=-1,
    \end{cases}       
\end{equation*}
where $c$ stands for the sectional curvature. In what follows we shall think of $ \mathbb{M}^n(c)$ as a hypersurface in the semi-Euclidean space $\mathbb{R}_{\nu}^{n+1}$. Now, two important comments are due. Firstly, the following general framework is well-known in the theory of Ricci almost solitons. Write $g_c$ for the standard metric of either of the cases above. Take a vector field $X$ on $\mathbb{M}^n(c)$ such that $X(x) = v^T(x)$, where $v \in \mathbb{R}^{n+1}_{\nu}$ is a constant  space-like unit vector, and let $h_v: \mathbb{M}^n(c) \rightarrow \mathbb{R}$ be the height function on $\mathbb{M}^n(c)$ corresponding to the vector $v$, given by $h_v(x) = g_c(\vec{x},v)$, for all $x \in \mathbb{M}^n(c)$. Therefore, $X(x)= \nabla h_v(x)$ and ${\mathrm{Hess}}\, h_v = -ch_v g_c$. It is known that the Ricci tensor of $g_c$ is given by $Ric = c(n-1) g_c$. Thus,  $(\mathbb{M}^n(c),g_c,\nabla h_v)$ is a gradient Ricci almost soliton with soliton function $\lambda_v = c(n-1) - ch_v$. Secondly, we know from \cite{Kim} that all the conformal vector fields on $\mathbb{M}^n(c)$ are given by	 
	 \begin{equation}\label{CampopotencialSn}
	 \overline{V} \vert_{\mathbb{M}^n(c)} = \Big[-\varepsilon_N \overline{g}( \vec{x} , \vec{\gamma}) \vec{x} + B\vec{x} + \vec{\gamma}\Big] \Big\vert_{\mathbb{M}^n(c)},
	\end{equation}
	for some vector field $\vec{\gamma}$ on $\mathbb{R}_{\nu}^{n+1}$. Notice that the Killing component $ B\vec{x}\vert_{\mathbb{M}^n(c)}$ of the vector field \eqref{CampopotencialSn} does not have a normal part. Notice that this latter equation is of paramount importance in our construction. With these two comments in mind, we are now in a position to elaborate on the subsequent examples.

\begin{exem}\label{exem2}
	 Let $\mathbb{M}^n(c)$ be as above and let $\overline{V}$ be a conformal vector field in $\mathbb{R}_{\nu}^{n+1}$. For all $x \in \mathbb{M}^n(c)$ take the normal unit vector field $N=\vec{x}$ to $\mathbb{M}^n(c)$ as the position vector, and  recall that  $\overline{g}(\vec{x},\vec{x})=\varepsilon_N$. As the vector field $ \overline{V} \vert_{\mathbb{M}^n(c)}$ is decomposed into the sum of its tangential part $V$ and its normal part, we clearly have for any $X\in \mathfrak{X}(M^n)$ that $\langle  \overline{V} \vert_{\mathbb{M}^n(c)}, X\rangle = \langle V, X\rangle$. Now it takes a simple calculation to show that the tangent component of  the vector field $ \overline{V} \vert_{\mathbb{M}^n(c)} $ is given by 
\begin{equation}\label{CampopotencialSn1}
V = \vec{\gamma}^T + B\vec{x}\vert_{\mathbb{M}^n(c)}.
\end{equation}	
  By dint of Theorem~\ref{thm1} it is $V$ that determines the Ricci almost soliton structure. More precisely, it is virtually $ \vec{\gamma}^T$ which will define the latter structure as the Killing part of $V$ will naturally step out. Writing $h_{\vec{\gamma}}(x) = \overline{g} (\vec{x},\vec{\gamma})$ for the height function in the direction of the vector $\vec{\gamma}$ we observe that $\nabla h_{\vec{\gamma}} = \vec{\gamma}^T$, which is, we indeed have a gradient Ricci almost soliton. In other words, $Ric + \frac{1}{2}\mathcal{L}_V g_c = \lambda g_c$ holds. Since $Ric = \varepsilon_N(n-1)g_c$ and $\mathcal{L}_Vg_c = - 2 \varepsilon_N h_{\vec{\gamma}} g_c$ we obtain for the soliton function 
 \begin{equation}\label{funcaosolitonSn}
\lambda_{\vec{\gamma}} = \varepsilon_N\Big(n-1-h_{\vec{\gamma}}\Big).
\end{equation} 
This way, we have shown that $\mathbb{M}^n(c)$ has a structure of a gradient Ricci almost soliton, naturally given by the conformal vector field~\eqref{CampopotencialSn1}.
\end{exem}

The construction of the next two examples is motivated by the following observations. Our quest for a Ricci almost soliton structure on the pseudo-hyperbolic spaces of zero or negative type is naturally motivated by Tashiro's theorem. Furthermore, in \cite{prrs}, Pigola et al. considered the warped product  $I\times_{\varphi}\Sigma$ of the real interval $I\subset\mathbb{R}$ with a Riemannian manifold $\Sigma$,  and warping function $\varphi:\,I\longrightarrow \mathbb{R}^{+}$. They proved under which conditions both $\Sigma$  and $I\times_{\varphi}\Sigma$ are  Einstein manifolds as well as that $I\times_{\varphi}\Sigma$  has a structure of a gradient Ricci almost soliton.%================================================================================================
%                                                                  EXAMPLE 3
%================================================================================================

\begin{exem}\label{exem3}
Consider the  pseudo-hyperbolic space $(\Bbb{R} \times_{e^t} \Bbb{R}^{n-1},g)$ of zero type as a hypersurface in the Lorentz space $(\Bbb{L}^{n+1},\overline{g})$. Note that $(\Bbb{R} \times_{e^t} \Bbb{R}^{n-1},g)$ is isometric to $(\Bbb{H}^n,g_{\Bbb{H}^n})$ via the isometry $G: (\Bbb{R} \times_{e^t} \Bbb{R}^{n-1},g) \longrightarrow (\Bbb{H}^n,g_{\Bbb{H}^n})$ given by $G(t,x)=(x,e^{-t})$. We derive the potential field $V_{\Bbb{R} \times_{e^t} \Bbb{R}^{n-1}}$ of $\Bbb{R} \times_{e^t} \Bbb{R}^{n-1}$ by taking the tangential projection of the conformal vector field $\overline{V}$ on $\Bbb{L}^{n+1}$ as prescribed by equation~\eqref{campoconfRn+1}. Thence, for any $y=(t,x) \in \Bbb{R} \times_{e^t} \Bbb{R}^{n-1}$, the potential field of $\Bbb{R} \times_{e^t} \Bbb{R}^{n-1}$ is given by
\begin{equation}\label{CampopotencialRetRn}
V_{\Bbb{R} \times_{e^t} \Bbb{R}^{n-1}}(y) = dG^{-1}(V_{\Bbb{H}^n}(G(y))).
\end{equation}
Bearing equation~\eqref{CampopotencialSn1} in mind, we perceive that equation~\eqref{CampopotencialRetRn} now reads
\begin{equation}\label{CampopotencialRetRn1}
V_{\Bbb{R} \times_{e^t} \Bbb{R}^{n-1}}(y) = dG^{-1}(\vec{\gamma}^T)+dG^{-1}(B\overrightarrow{G(y)}).
\end{equation}
Since $G$ is an isometry, by Example~\ref{exem2}, we have that
\begin{equation*}
g=G^{*} g_{\Bbb{H}^n},
\end{equation*}
\begin{equation*}
Ric_{g}=G^{*} Ric_{\Bbb{H}^n} = -(n-1)G^{*} g_{\Bbb{H}^n}=-(n-1)g,
\end{equation*}
\begin{equation*}
\mathcal{L}_{V_{\Bbb{R} \times_{e^t} \Bbb{R}^{n-1}}}g=G^{*} \mathcal{L}_{V_{\Bbb{H}^n}}g_{\Bbb{H}^n} = 2 (h_{\vec{\gamma}} \circ G) G^{*}g_{\Bbb{H}^n} =  2 (h_{\vec{\gamma}} \circ G) g,
\end{equation*}
\begin{equation}\label{funcaosolitonRetRn}
\lambda_{\Bbb{R} \times_{e^t} \Bbb{R}^{n-1}} = \lambda_{\vec{\gamma}} \circ G = -(n-1)+(h_{\vec{\gamma}} \circ G).
\end{equation}
Thus, $(\Bbb{R} \times_{e^t} \Bbb{R}^{n-1},g)$ has a structure of a Ricci almost soliton with potential field given by~\eqref{CampopotencialRetRn1} and a soliton function given by \eqref{funcaosolitonRetRn}.
\end{exem}

%================================================================================================
%                                                                  EXAMPLE 4
%================================================================================================

\begin{exem}\label{exem4}
	Consider now the pseudo-hyperbolic space $(\Bbb{R} \times_{cosh t} \Bbb{H}^{n-1},g)$ of negative type as a hypersurface in the  Lorentz space $(\Bbb{L}^{n+1},\overline{g})$.  We know that $(\Bbb{R} \times_{cosht} \Bbb{H}^{n-1}, g)$ isometric to $(\Bbb{H}^n,g_{\Bbb{H}^n})$  via the isometry $P: (\Bbb{R} \times_{cosht} \Bbb{H}^{n-1},g) \longrightarrow (\Bbb{H}^n,g_{\Bbb{H}^n})$ given by $P(t,x)=(x,cosh^{-1}t)$. We construct the potential field $V_{\Bbb{R} \times_{cosht} \Bbb{H}^{n-1}}$ on $\Bbb{R} \times_{cosht} \Bbb{H}^{n-1}$ by taking the tangential projection of the conformal field $\overline{V}$ of $\Bbb{L}^{n+1}$ as given in equation~\eqref{campoconfRn+1}. We have that, for all $y=(t,x) \in \Bbb{R} \times_{cosht} \Bbb{H}^{n-1}$, the potential field of $\Bbb{R} \times_{cosht} \Bbb{H}^{n-1}$ is given by
	\begin{equation}\label{CampopotencialRcoshtHn}
	V_{\Bbb{R} \times_{cosht} \Bbb{H}^{n-1}}(y) = dP^{-1}(V_{\Bbb{H}^n}(P(y))).
	\end{equation}
	Again, in the spirit of equation~\eqref{CampopotencialSn1}, equation~\eqref{CampopotencialRcoshtHn} now reads
	\begin{equation}\label{CampopotencialRcoshtHn1}
	V_{\Bbb{R} \times_{cosht} \Bbb{H}^{n-1}}(y) = dP^{-1}(\vec{\gamma}^T)+dP^{-1}(B\overrightarrow{P(y)}).
	\end{equation}
 As $P$ is an isometry, by  Example~\ref{exem2}, we have that
	\begin{equation*}
	g=P^{*} g_{\Bbb{H}^n},
	\end{equation*}
	\begin{equation*}
	Ric_{g}=P^{*} Ric_{\Bbb{H}^n} = -(n-1)P^{*} g_{\Bbb{H}^n}=-(n-1)g,
	\end{equation*}
	\begin{equation*}
	\mathcal{L}_{V_{\Bbb{R} \times_{cosht} \Bbb{H}^{n-1}}}g=P^{*} \mathcal{L}_{V_{\Bbb{H}^n}}g_{\Bbb{H}^n} = 2 (h_{\vec{\gamma}} \circ P) P^{*}g_{\Bbb{H}^n} =  2 (h_{\vec{\gamma}} \circ P) g,
	\end{equation*}
	\begin{equation}\label{funcaosolitonRcoshtHn}
	\lambda_{\Bbb{R} \times_{cosht} \Bbb{H}^{n-1}} = \lambda_{\vec{\gamma}} \circ P = -(n-1)+(h_{\vec{\gamma}} \circ P).
	\end{equation}
 Thus, $(\Bbb{R} \times_{cosht} \Bbb{H}^{n-1},g)$ has a structure of a Ricci almost soliton with potential field given by~\eqref{CampopotencialRcoshtHn1} and a soliton function  given by \eqref{funcaosolitonRcoshtHn}.
\end{exem}

%============================================================================================
%                                              A CLASSIFICATION THEOREM                      
%============================================================================================
\section{A Classification Theorem}\label{conclusions}

We have now reached the apex of this paper, which is a classification of a class of Ricci almost solitons generated by a conformal vector field. We already stated this latter theorem in the introduction, and before we perceive its validity, we should like to briefly comment on how our work relates to the results in \cite{Pacific}. By so doing, we hope yet again that the reader would clearly understand the sheer motivation of the present work. To do so in the most elegant and concise form, we venture to summarise the results of the latter paper in a single theorem. For this purpose, we shall write  $\bf {{H}^{n+1}}$ for either the hyperbolic space $\mathbb{H}^{n+1}$, the de Sitter space $\mathbb{S}^{n+1}_1$ or anti-de Sitter space $\mathbb{H}^{n+1}_1$, and $\mathbb{R}^{n+2}_{\nu}$ for the semi-Euclidean space with signature $\nu=1,2$.  
\begin{thmc}\label{Pacific1}[Aquino, de Lima, Gomes~\cite{Pacific}]
Let $\psi: \Sigma^n\longrightarrow \bf{H^{n+1}}$ be a spacelike hypersurface immersed in $\bf{H^{n+1}}$. Suppose that for some nonzero vector $a\in\mathbb{R}^{n+2}_{\nu}$, the vector field $a^\top$ provides the structure of a gradient Ricci almost soliton for $\Sigma^n$. If the image of the Gauss mapping of $\Sigma^n$ lies in a totally umbilical (spacelike) hypersurface of $\bf{H^{n+1}}$ determined by $a$, then $\Sigma^n$ is a totally umbilical hypersurface of $\bf {H^{n+1}}$. 
\end{thmc}
It ought to be noticed that, in the original paper, this latter theorem constitutes three theorems, each of which is neatly proven. To avoid misconceptions, we must make the following remark. In each of the original theorems, the image of the Gauss mapping of $\Sigma^n$ is supposed to  lie in an ambient space that is not necessarily the same as the ambient space in which $\Sigma^n$ is itself immersed! More precisely, the hypotheses of Theorem 4 in \cite{Pacific} require that $\Sigma^n$ be hypersurface immersed in  $\mathbb{H}^{n+1}$, and the image of the Gauss mapping of  $\Sigma^n$ lies in a totally umbilical spacelike hypersurface of $\mathbb{S}^{n+1}_1$.  Conversely, the hypotheses of Theorem 6 in \cite{Pacific} require that $\Sigma^n$ be a spacelike hypersurface immersed in $\mathbb{S}^{n+1}_1$, and the image of the Gauss mapping of  $\Sigma^n$ lies in a totally umbilical hypersurface of $\mathbb{H}^{n+1}$. In the case of Theorem 8 in \cite{Pacific} it is supposed that $\Sigma^n$ is spacelike hypersurface of $\mathbb{H}_1^{n+1}$  and the image of the Gauss mapping of  $\Sigma^n$ lies in a totally umbilical hypersurface of $\mathbb{H}_1^{n+1}$. All these latter hypotheses are important for the method of proof in \cite{Pacific}, and the reader might care to take a peek at the proofs. Even without such an effort, however, we believe that it is not difficult to observe at this juncture that our Theorem~\ref{thm1} is somewhat the converse of Theorem~C and, indeed, more general in nature. In our case, in contrast to the hypotheses in \cite{Pacific}, we substantially assume that the hypersurface $M^n$ is totally umbilical. We then prove that the Ricci almost soliton structure on $M^n$ is necessarily and sufficiently determined by the tangential part of a conformal vector field on the ambient space. 

Another interesting and important result in \cite{Pacific} is a corollary of Theorem~C.  Notice that in the original paper, it again constituted as three distinct corollaries.
\begin{cor}\label{Pacific2}[Aquino, de Lima, Gomes~\cite{Pacific}]
If $\Sigma^n$ is a complete spacelike hypersurface of $\bf{H^{n+1}}$ such that for some nonzero vector $a\in\mathbb{R}^{n+2}_\nu$, the vector field $a^\top$ provides on it a nontrivial structure of a gradient Ricci almost soliton  and the image of its Gauss mapping lies in a totally umbilical spacelike hypersurface of $\bf{H^{n+1}}$ determined by $a$, then $\Sigma^n$ is isometric to either $\mathbb{S}^n$ or $\mathbb{H}^n$  (depending on  $\bf{H^{n+1}}$) .  
\end{cor}
Now, bearing this latter discussion in mind, let us swiftly remind ourselves what has been done in this article. Firstly, by dint of Theorem \ref{thm1}, it was perceived that the Ricci almost soliton structure on a totally umbillical hypersurface isometrically immersed into a semi-Riemannian manifold is a natural manifestation of the existence of a conformal vector field on the ambient space.  Secondly, it was shown that Theorem \ref{thm1} readily manifests itself  into Examples \ref{exem1}, \ref{exem2},  \ref{exem3}, and \ref{exem4}. In other words, we have tacitly proven the converse of the above corollary. Thirdly, we saw in Theorem \ref{thm2}, which conditions guarantee that a complete Riemannian manifold isometrically immersed into a semi-Riemannian manifold and already endowed with a Ricci almost soliton structure by virtue of Theorem \ref{thm1} necessarily admits a special concircular field. To put it differently, Theorem \ref{thm2} enables us to invoke  Tashiro's theorem. Ergo, we have just perceived the truth of the classification theorem announced in the introduction. Namely, the following classification theorem holds.

\begin{theorem*}
Let $(M^n,g)$ be a complete Riemannian manifold isometrically immersed into a  semi-Riemannian manifold $(\overline{M}^{n+1}, \overline{g})$ of constant sectional curvature endowed with a conformal vector field $\overline{V}$.  Let  $\sigma$ be the conformal factor of $\overline{V}$ and $\lambda$ be the soliton function. Suppose that $\psi=\sigma|_{M^n}-\lambda$ is non-zero in some dense subset of $M^n$ and $Ric = \mu \mathcal{A}$, for some smooth function $\mu$ on $M^n$. Then, the Ricci almost soliton arising from the conformal vector field $\overline{V}$ is one of the following manifolds: 
\newline\\
A)\,\, an Euclidean space $\mathbb{R}^n$;
\newline\\
B)\,\, a sphere $\mathbb{S}^n$;
\newline\\
C)\,\, a hyperbolic space $\mathbb{H}^n$;
\newline\\
D)\,\, a pseudo-hyperbolic space of zero or negative type.
\end{theorem*}
 
%The demanding reader may ask at this point what would conceivably happen if our pillar hypothesis of umbilicity is dropped out? This is indeed a natural question that we plan to investigate in the nearest future.

\section*{Acknowledgements}

The first author is partially supported by Conselho Nacional de Desenvolvimento Científico e Tecnológico (CNPq), of the Ministry of Science, Technology and Innovation of Brazil (Grants 428299/2018-0 and 307374/2018-1). The second author is wholeheartedly acknowledging the financial support by Funda\c{c}\~ao de Amparo \`a Pesquisa do Estado do Amazonas -  FAPEAM (Grant 062.00549/2019) in terms of a doctoral scholarship. The third author  is also partially supported by CNPq (Grant  428299/2018-0), and is humbly grateful for all  vicissitudes and tribulations of life. We are also indebted to $\cdots\cdots\cdots\cdots\cdots\cdots$
, for their generosity and constructive criticism.

\end{document}